\documentclass[a4paper,10pt]{article}
\usepackage{amsmath,amssymb,graphicx}

\newtheorem{theorem}{Theorem}[section]
\newtheorem{lemma}[theorem]{Lemma}
\newtheorem{proposition}[theorem]{Proposition}

\newtheorem{remark}[theorem]{Remark}

\newcommand{\uD}{\tfrac{\partial u}{\partial D}}

\newcommand{\pa}[1]{\tfrac{\partial}{\partial #1}}

\newcommand{\R}{\mathbb{R}}
\newcommand{\N}{\mathbb{N}}
\newcommand{\Nd}{N_{data}}

\title{Optimization of the shape (and topology) of the initial conditions
for diffusion\\ parameter identification}
\author{Stefan Kindermann\footnotemark[2]
and \v St\v ep\' an Pap\' a\v cek\footnotemark[3]}

\date{}
 
\begin{document}
\maketitle
\footnotetext[2]{Corresponding author. E-mail: {kindermann@indmath.uni-linz.ac.at}.
Industrial Mathematics Institute, University Linz, Altenbergerstr. 69, 4040 Linz,  Austria.}  
\footnotetext[3]{Institute of Complex Systems, University of 
South Bohemia in \v{C}esk\'{e}
Bud\v{e}jovice, \\FFPW USB, CENAKVA, Z\' amek 136, 373 33 Nov\' e
Hrady, Czech Republic.}

\begin{abstract}
The design  of an experiment, e.g., the setting of initial
conditions, strongly influences the accuracy of the whole process of
determining  model parameters from data.  
%We are aware of the
%potential benefit in parameter identification when the experimental
%design variables are chosen optimally. Thus, 
We impose a
sensitivity-based approach for choosing optimal design variables
and study the 
optimization of the shape (and
topology) of the initial conditions for an inverse problem of 
a diffusion parameter identification. Our approach, although case
independent, is illustrated at the FRAP (Fluorescence Recovery After
Photobleaching) experimental technique.
 The core
idea resides in the maximization of a sensitivity measure, which
depends on a specific experimental setting of initial conditions.
By a numerical optimization, we find an interesting pattern of increasingly complicated
(with respect to connectivity) optimal initial shapes. 
 The proposed modification of the FRAP experimental protocol
  is rather surprising but entirely realistic and the resulting enhancement of the
parameter estimate accuracy is significant.

\end{abstract}
%\vspace{2ex}

%


\section{Introduction}

 The common practice
of  setting  experimental conditions resides on trial-error method
performed by experimentalists while the subsequent data processing
is not always taken into account.
It is not a rare case that large amount of data is routinely
generated without a clear idea about further  data processing.
Here, we suggest to analyze simultaneously both the data (i.e., the
processes hidden in data) and the experimental protocol, aiming to
establish the link between experimental conditions and the accuracy
of our results.
 The whole idea is presented in a simplified case study of Fluorescence Recovery 
 After Photobleaching (FRAP) data processing. It serves as a paradigmatic 
 example of the inverse problem of identifying the diffusion parameter 
 from spatio-temporal concentration measurements.

 FRAP is a classical
method used to study the mobility of fluorescent mo\-le\-cules
%(presumably due to the diffusion)
in membranes of the living cells \cite{Florian10}.
The FRAP   technique is based on measuring the fluorescence
intensity, which is proportional to the non-bleached particles
concentration  in response to a high-intensity laser pulse.
We suppose  the laser pulse (the  \emph{bleach}) causes an
irreversible loss in fluorescence  of a certain amount of particles
originally in the bleached area.
The monitored region or the region of interest (ROI) where the data
are measured is usually an Euclidean 2D domain containing the
bleached area.
 After the {bleach},  the  change in
fluorescence intensity  in a monitored region is observed due to the
transport of  fluorescent compounds from the area outside the bleach
as well as bleached particles from the bleached region to originally
non-bleached regions. In general,
 we observe both recovery   and loss
in fluorescence  in different regions corresponding to FRAP or FLIP
(Fluorescence Loss in Photobleaching), respectively
\cite{Kana14,IWBBIO15}.
The natural question to be asked is: How does the bleach shape (and
topology) influence the accuracy of resulting parameter
estimates\footnote{The problem is partially solved in our paper
\cite{AAA15} where we use a relation similar to \eqref{conf}
assuring smaller confidence intervals of the parameter estimates for
the higher values of a sensitivity measure.}.
Therefore, the main focus of this study concerns the searching for
the optimal bleach shape (and its topology), or, from the more mathematical viewpoint,
to optimize  binary-valued initial conditions in a diffusion-parameter identification problem 
with respect to sensitivity.

 The rest of this article is organized as follows: In Section 2 we provide the preliminaries for the 
 further optimization problem
 formulation  as well as
  the background information concerning the FRAP method.
 In Section 3 we rigorously formulate the optimization problem, evaluate the key term in the objective function and announce the Proposition~3.2, which allows to solve the problem efficiently. Section 4   provides two numerical results: (i) set of solutions (depending on the characteristic time of diffusion) for an optimal
   setting of the  bleached region shape and topology when the bleach depth is the only constraint,
   and (ii) corresponding results as in (i) when the total bleach energy is restricted as well.
The  novelty and  benefits of our approach as well as outlooks for
further research are resumed in the final Section 5.

\section{Preliminaries}
\subsection{Inverse problem of model parameter identification from FRAP data}

  Based on the spatio-temporal FRAP data,  the so-called \emph{effective diffusion coefficient}
 was  estimated using  a {closed form model}
\cite{Axel76,Moul97,Ellenberg} in past decades.
Nowadays, numerical simulations are preferred, cf.
\cite{IvoF,MCM2013,AAA15,Mai11,NatureSR15}, mainly because there is
no need for some unrealistic assumptions.

 Further, we consider a normal Fickian diffusion problem with a
{\em constant} diffusion coefficient $D>0$. In the setup of this
paper, we assume a sufficiently large domain such that we can treat
diffusion in the free space $\R^2.$ In FRAP, the simplest governing
equation for the spatio-temporal distribution of fluorescent
particle concentration $u(x,t)$ is a diffusion equation without
reaction term as follows:
\begin{align}
\pa{t} u(x,t) &= D \Delta u(x,t) & & x \in \R^2, t \in [0,T], \label{one} \\
u(x,0) &= u_0(x)& &  x \in \R^2.  \label{two}
\end{align}

The main issue in FRAP and related identification problems is to
find the value of the diffusion coefficient $D$ from spatio-temporal
measurement of the concentration $u(x,t)$. As another simplification, we assume a spatially radially
symmetric observation  domain, that is, we consider the data
\begin{equation}\label{data}
\begin{split} \text{data}&:= u(x,t) \qquad x \in \Omega  \times [0,T], \\
\Omega &=  \{ x \in \R^2\,|\, \|x\| \leq R \}.
\end{split}
\end{equation}
Hence, the data are observed on a cylinder with radius $R$ and
height $T$. The parameters $R,T>0$ are fixed for the further
analysis and not subject of optimization. Although, in practice, it
is often more convenient to consider an square spatial domain (i.e.,
an image), our cylindrical domain is a reasonable general model to
indicate the most important conclusions in our further analysis.

The data \eqref{data} are insofar idealized as  
in experiments, they are given only at discrete points $u(x_i,t_i)$, $i, 1,\ldots N_{data}$. 
However, for simplicity and as it was done in \cite{AAA15}, we assume a sufficiently
dense set of data points $(x_i,t_i)$ such that the model \eqref{data} is valid with 
reasonable accuracy  and 
such that certain sums over the data points can be approximately well by 
the corresponding  integrals over  $\Omega  \times [0,T];$ see \cite{AAA15}.

The equations~\eqref{one}--\eqref{two} and the data \eqref{data} are the basis for all
the subsequent analysis, hence, although we mentioned FRAP as one important application, 
our results are of course applicable to other diffusion parameter identification problem
which are governed by similar equations and data assimilation processes.

Based on the parameters $R,T$, it is convenient to introduce the following scaling of
the space and time coordinates:
\begin{equation} \label{scaling}
\begin{split}
&z:= \frac{x}{R}, \quad  \tau := \frac{t}{T}, \\
&v(z,\tau) := u(z R ,\tau t) \Leftrightarrow  u(x,t) =
v(\tfrac{x}{R},\tfrac{t}{T}) \quad v_0(z) :=u_0(z R).
\end{split}
\end{equation}
For later use we define a scaled version of the inverse
of the diffusion coefficient\footnote{For the characteristic
diffusion time  $ t_c = \frac{R^2}{4 D} $, we have $ \beta =
\frac{t_c}{T} $.}
\begin{equation}\label{beta} \beta := \frac{R^2}{4 T D} \, \cdot \end{equation}
The scaled concentration $v$ satisfies the equation
\begin{align}
 \pa{\tau} v(z,\tau) &=  \tfrac{T D}{R^2}   \Delta v(z,\tau)
 =\tfrac{1}{4 \beta}  \Delta v(z,\tau)
 & & z \in \R^2, \tau \in [0,1] ,\label{onesc} \\
 v(z,0) &= v_0(z)= u_0(z R) & &  z \in  \R^2.  \label{twosc}
\end{align}
In our case of constant coefficients and free-space diffusion in two dimensions, the solution to this problem can be
expressed by means of the Green's function $G(x,t;y)$ for the heat equation.
\begin{equation}\label{trans} u(x,t) =
 \int_{\R^2} G(x-y,t D) u_0(y) dy,
 \qquad v(z,t) =  \int_{\R^2} G(z-y,\tfrac{\tau}{4 \beta} ) v_0(y) dy
\end{equation}
with   the well-known two-dimensional heat kernel
\begin{align*}
      G(z,t) &:= \tfrac{1}{ 4 \pi t} e^{-\frac{\|z\|^2}{4t}}\, .
\end{align*}

\subsection{Sensitivity of the output observation with respect to model parameter}
Given the data as above, the diffusion coefficient $D$ can be
inferred by fitting the solution given in \eqref{trans} to the data
on $\Omega \times [0,T]$. Because of unavoidable  noise in the data,
one obtains an estimated value $\overline{D}$ which  reasonably well
approximates the true $D$.
It can be shown, see  \cite{AAA15,Bates} and references within
there, that for our case of single scalar parameter estimation and assuming white noise as 
data error, the
expected relative error in $D$ depends on the data noise and a
factor, which we call relative global sensitivity $S_{GRS},$ as follows:
\begin{equation}\label{conf}
 \mathbb{E} \left(\left|\frac{\overline{D}-D}{D}\right|^2 \right)  
 \sim \frac{1}{S_{GRS}} \frac{\sigma ^2}{{u_{{\rm ref}}}^2} ,
\end{equation}
where $\sigma^2$ denotes the noise variance (and
$\frac{\sigma}{u_{{\rm ref}}}$ is related to the coefficient of variation or
to the inverse of the signal to noise ratio), and 
$u_{{\rm ref}}$ is some (chosen) reference value of the observed output.

The global relative squared sensitivity is given by
\begin{equation}\label{sensi}
 S_{GRS} := \frac{D^2}{u_{{\rm ref}}^2} \sum_{i=1}^{\Nd} (\pa{D} u(x_i,t_i))^2,
\end{equation}
where
 $\uD(x_i,t_i)$ is the usual sensitivity of the output
 observation data point $(x_i,t_i)$ with respect to the parameter $D$ and
 $\Nd$ is the number of data points in space-time domain.
It is obvious from this estimate that if the noise level is fixed,
the estimation can be only improved by switching to an experimental
design with a higher sensitivity. Thus, the thrive for good
estimators $\overline{D}$ leads to the problem of experimental
design optimization in order to maximize the sensitivity measure
$S_{GRS}$. The solution of this optimization problem, which is the
central aspect of this article, is rigorously formulated and
partially solved in the next Section 3.

\section{Optimization of the bleach design}
\subsection{Problem formulation}
The sensitivity measure \eqref{sensi} involves several design
parameters. Note that $R$ and $T$ are involved implicitly because  the number of data
points  $\Nd \approx \frac{\pi R^2 T}{\Delta x \Delta t}$, where
$\Delta x$ is defined by the pixel size and $\Delta t$ corresponds
to the time interval between two consecutive measurements.
However, if all the above parameters $R,T, \Delta x, \Delta t$ are fixed,
there is only one way to maximize the sensitivity measure $S_{GRS}$:
 considering the {\em initial bleach} $u_0$ in \eqref{two} as the
experimental design parameter. By optimizing the bleach design, we
mean to select the initial conditions in such a way that $S_{GRS}$
is maximized and hence the expected error in $D$ minimized. In order
to do so, we have to chose the class of designs out of which we take
the initial conditions. In FRAP, it is the usual  case that the
initial bleach $u_0$ is a binary-valued function with fixed given
values $u_0(x) \in \{u_{00},0\},$ $u_{00}>0$. Without loss of
generality, we assume $u_0(x)$ being  a
$\{1,0\}$-function, which also cancels out the term $u_{{\rm ref}}$ in $S_{GRS}$\footnote{When we choose
$u_{{\rm ref}} = u_{00}$, then the
normalization of the signal, i.e. $0 \leq u(x) \leq 1$ is actually
done due to the division by the maximal value which signal $u$ can
reach, cf. \eqref{sensi}.}:
\[ u_0(x) = \begin{cases} 1  & x \in B \\ 0, & \text{ else, }\end{cases}\]
with some bounded open set $B$, the {\em bleach shape}.
Moreover, if the sum in
the term $S_{GRS}$ is furthermore  approximated by an integral
\[
 \sum_{i=1}^{\Nd} (\pa{D} u(x_i,t_i))^2  \sim \frac{\Nd}{\pi R^2 T}
 \int_{\|x\|\leq R} \int_0^T  |\pa{D} u(x,t)|^2 dx dt,
 \]
 then we have a final sensitivity term $S_{int}$ to be maximized:
\begin{equation}\label{sensii}
S_{int} := \int_{\|x\|\leq R} \int_0^T |\uD(x,t)|^2 dx d t.
\end{equation}

Depending on the different restrictions imposed on the initial bleach,
we study the following two setups:

\begin{itemize}
\item Problem 1: optimization with fixed bleach depth
\[  S_{int} \to \max_{B \subset \R^2}  \qquad \text{with}   
\quad u_0(x) = \begin{cases} 1  & x \in B \\ 0 & \text{ else }\end{cases}, \]
over the set of bounded open sets $B \subset \R^2.$
\item Problem 2: optimization with fixed bleach depth and fixed energy
\[   S_{int} \to \max_{B \subset \R^2} \qquad \text{ with} 
\quad  u_0(x) = \begin{cases} 1  & x \in B \\ 0 & \text{ else }\end{cases}
 \ \text{ and }  \ \int_{\R^2} u_0(x) dx = c_1,
\]
$c_1>0$ given,  over the set of bounded open sets $B \subset \R^2.$
\end{itemize}
Note that these problems are  shape (and topology) optimization
problems, i.e., the unknown bleach shape $B$ is the independent
variable. 

However, the problems can be simplified by restricting $B$ to a radially symmetric shapes; the
reason for that is the following lemma.
\begin{lemma}\label{hh}
If solutions to the Problems 1 and 2 exists, then there also exists
 radially symmetric solutions. In particular, if the solutions are unique, they must be radially symmetric.
\end{lemma}
\begin{proof}
It is easy to see that the sensitivity measure $S_{int}$ is a convex
functional of the initial conditions.
%If a bounded domain exists as solution, then we may include the additonal constraint that
%$D \subset \{ x \in \R^2 \, |\, \|x\| \leq R_2 \}$ with a sufficiently large chosen radius $R_2,$ and this
%does not alter the solutions of the optimization problem.
From convex analysis, we conclude that
the optimization problems can be relaxed and expressed as problems with box constraints.
In either case, the set of initial conditions can be  relaxed to
\begin{equation}\label{eq:rel}  u_0 \in \{ w(x) \,| \, 0 \leq w(x) \leq 1
\},
\end{equation}
without altering the solution. Indeed, this holds because a convex functional attains its maximum at the
boundary, hence, by considering the relaxed constraints \eqref{eq:rel}, a solution will be at the boundary of
the admissible domain in \eqref{eq:rel} and hence a $\{0,1\}$-function. Thus, solutions of the relaxed
problem are also solutions of the original problem (and vice versa).

Due to our choice of an  observation domain, the
problem is invariant with respect to coordinate rotations, i.e., the
sensitivity does not change if the initial bleach is rotated. If
$u_0(x)$ is a solution to the optimization problem, then so is
$u_{0}(Q_\theta x)$ for any rotation $Q_\theta$  around the origin
with angle $\theta$. It is easy to see that taking the
angle-averaged initial condition $\frac{1}{2 \pi} \int_0^{2 \pi}
u_0(Q_\theta z) d \theta$ yields  a solution to the relaxed
optimization problem which is radially symmetric. Because it must
also be  a solution to the original problem, the set $B$ must be
radially symmetric. $\blacksquare$
\end{proof}

Note that Lemma~\ref{hh} is not valid if the observation domain is not radially symmetric.
When speaking to experimentalists, the result of Lemma~\ref{hh} is usually intuitively 
clear to them. However, a first guess for the optimal shape that is often uttered by them, is 
that it must  be a ``disk'' (or a ball). It is maybe the main conclusion of our work that 
this guess not always hits the truth  as the optimal shape can be an annulus or multiple annuli as well; see below.

\begin{remark}\label{rem2.2}
{\rm Let us mention  several related design optimization problems.
In a previous paper \cite{AAA15}, the authors have considered the
problem of optimizing the sensitivity with respect to the
observation domain $R,T$ and also the problem of selecting reduced
data that only little reduce the sensitivity $S_{GRS}$. 

In view of
the initial conditions, there are other classes of designs possible,
for instance, if not binary-valued functions are used, then it is
natural to optimize sensitivity with respect to a fixed energy
(i.e., the $L^1$-norm) of the initial conditions. That is, the
problem
 \[  S_{int} \to \max_{u_0} \qquad \text{ with}   \quad \int_{\R^2} u_0(x) dx = c_1  \]
with  $c_1>0$ given.  However, when analyzing this problem, one
finds out that it does not have a solution. A sequence of almost
optimal solutions will tend to a $\delta$-distribution. Thus a
$\delta$-peak is the ``optimal'' solution in this case. The
sensitivity is infinite for this case, which is caused by the fact
that the solutions are not square-integrable any more. Thus, for a
$\delta$-peak, one cannot use the formula above for the expected
error but have to used different, weighted norms (for the data noise
and the sensitivity).

A similar problem is that one with a fixed $L^2$-norm.
 \[  S_{int} \to \max_{u_0} \qquad \text{ with}  \int_{\R^2} u_0(x)^2 dx = c_2 \]
This one has an appealing interpretation as eigenvalue problem:
indeed, let $K$ be the linear operator $L^2(\R^2) \to L^2(\Omega
\times [0,T])$ that maps the initial conditions to $\uD$ on the
observation domain. It can be shown that this is a bounded linear
and compact operator. The problem can be rephrased in functional
analytic language as
\[ \|K u_0\|_{L^2(\Omega \times [0,T])}^2  \to \max \qquad \text{with } \|u_0\|_{L^2(\R^2)} \leq c_2. \]
It is well known that a solution exists and it is given by the
right singular function  associated
to the largest singular value of $K$. The associated optimization problem thus can
be (approximately) solved by applying a standard  singular valued decomposition to the
discretized operator $K$.
}
\end{remark}

Going back to the optimization problems, Problems~1 and 2, it can be observed
that the sensitivity $S_{int}$ is a quadratic functional with
respect to $u_0$. For the optimization it is important to have an
efficient formula for $S_{int}$ available, since it has to be
evaluated many times (for different initial conditions). Solving the
PDE \eqref{one}--\eqref{two} is not convenient in that respect. We
rather stick to the integral representation by the Green's function
and derive a formula for the sensitivity $S_{int}$ in the next
section.

\subsection{Evaluating the sensitivity measure \eqref{sensii}}

The formula for  $S_{int}$ involves the derivative of the observation
$u(x,t)$ with respect to the diffusion parameter $D$. Either by differentiating
\eqref{one}--\eqref{two} or \eqref{trans}, we find that
\begin{align*} \uD(x,t) &= \tfrac{t}{D} \pa{t}u(x,t)  =
 t \Delta u(x,t)    =  t \Delta_x \int_{\R^2} G(x-y,t D) u_0(y) dy.
\end{align*}
Because a convolution integral commutes with differentiation, we
also find that
\begin{align*} \uD(x,t) &
= t   \int_{\R^2} G(x-y,t D) \Delta u_0(y) dy.
\end{align*}
In a similar way, we can derive a formula using the scaled
variables:
 \begin{align}\label{udsc}  \uD(z R, \tau T) &=
 \tfrac{T}{R^2} \tau  \int_{\R^2} G(z-y,\tfrac{\tau}{4 \beta}) \Delta v_0(y) dy =
  \tfrac{T}{R^2} \tau \Delta v(z,\tau)
 \end{align}
with $v_0(y) = u_0(R y).$

 Now, in order to solve the above  optimization
problems (Problems 1 and 2), we have to evaluate \eqref{sensi} for
many different initial conditions. Therefore, in the sequel we
derive the kernel of the quadratic form associated to the
sensitivity measure $S_{int}$.

Consider two solutions $u^{(1)}$ and $u^{(2)}$ of
\eqref{one}--\eqref{two} with the respective initial conditions
$u_0^{(1)}$ and $u_0^{(2)}.$  We associate to them the scaled
functions $v^{(1)}$ and $v^{(2)}$ as in \eqref{scaling}. We
calculate the following pairing of $u^{(1)}$ and $u^{(2)}:$
\begin{align*}
&\int_{\|x\|\leq R} \int_0^T  \tfrac{\partial u^{(1)}}{\partial D}(x,t)
 \tfrac{\partial u^{(2)}}{\partial D}(x,t) dx dt  \\
& = T R^2 \int_{\|z\|\leq 1} \int_0^1
 \tfrac{\partial u^{(1)}}{\partial D}(z R,\tau T)
 \tfrac{\partial u^{(2)}}{\partial D}(z R,\tau T) dz d \tau  \\
& = T R^2 \int_{\|z\| \leq 1} \int_0^1  \tau^2
\Delta v^{(1)}(z  ,\tau ) \Delta v^{(2)}(z ,\tau ) dz d \tau  \\
& =
\frac{T^3}{R^2}
\int_{\|z\| \leq 1} \int_0^1
   \int_{\R^n}  \int_{\R^n}    \tau^2 G(z-y, \tfrac{\tau}{4 \beta})  G(z-w,\tfrac{\tau}{4 \beta}) \Delta v_0^{(1)}(y)
   \Delta v_0^{(2)}(w)  dw dy d \tau d z \\
& =
\frac{T^3}{R^2}
 \frac{\beta^2}{ \pi^2}
 \int_{\|z\| \leq 1} \int_0^1
   \int_{\R^n}  \int_{\R^n}
 e^{ - \beta \frac{\|z - y\|^2 + \|z - w\|^2 }{\tau}}
 \Delta v_0^{(1)}(y)
   \Delta v_0^{(2)}(w)  dw dy d \tau d z.
\end{align*}

In the next step we assume  radially symmetric initial conditions and introduce polar coordinates:
\begin{align*}
 y &= r (\cos(\theta_1),\sin(\theta_1)) \quad   w = s (\cos(\theta_2),\sin(\theta_2))
\quad   z = q (\cos(\phi),\sin(\phi) \\
v_0^{(1)}(y) &=:  g^{(1)}(\|y\|) = g^{(1)}(r) \quad v_0^{(2)}(w) =:  g^{(2)}(\|w\|) = g^{(2)}(s)
\end{align*}
such that
\begin{align*}
    \Delta v_0^{(1)}(y) &=
\tfrac{1}{r} \tfrac{\partial}{\partial r} \left( r \tfrac{\partial}{\partial r}  g^{(1)}(r) \right) \qquad r = \|y\|,  \\
    \Delta v_0^{(2)}(w) &=
\tfrac{1}{s} \tfrac{\partial}{\partial s} \left( s \tfrac{\partial}{\partial s}  g^{(2)}(s) \right) \qquad s = \|w\|.
\end{align*}
Thus,
\begin{align*}
&\int_{\|x\|\leq R} \int_0^T  \tfrac{\partial u^{(1)}}{\partial D}(x,t)
 \tfrac{\partial u^{(2)}}{\partial D}(x,t) dx dt \\
& = \frac{T^3}{R^2}
 \frac{\beta^2}{ \pi^2}
 \int_0^1 q  \int_0^1
   \int_{\R}  \int_{\R}
   \int_0^{2 \pi} \int_0^{2\pi} \int_0^{2\pi}
 e^{ -\frac{\beta}{\tau}\left( q^2 + r^2 - 2 q r \cos(\phi-\theta_1) +   q^2 + s^2 -
 2 q s \cos(\phi-\theta_2)
 \right)} \\
 & \qquad \qquad {rs }
 \frac{1}{r} \frac{\partial}{\partial r} \left( r \frac{\partial}{\partial r}  g^{(1)}(r) \right)
   \frac{1}{s} \frac{\partial}{\partial s} \left( s \frac{\partial}{\partial s}  g^{(2)}(s) \right)
   d \theta_1 d \theta_2  d\phi dr  ds d \tau d q\, .
\end{align*}
By a substitution of $\phi-\theta_1$ by $\theta_1$ and
$\phi-\theta_2$ by $\theta_2$, the integrand does not depend on $\phi$. Thus, the
corresponding integral gives a contribution of $2 \pi.$
The integrals over $\theta_i$ can be calculated explicitly using the formula
\begin{equation}\label{bessel}  \int_0^{2 \pi} e^{a \cos(t)} d t = 2 \pi I_0(a),  \end{equation}
with the modified Bessel function $I_k(z).$
\begin{align*}
&\int_{\|x\|\leq R} \int_0^T  \tfrac{\partial u^{(1)}}{\partial D}(x,t)
 \tfrac{\partial u^{(2)}}{\partial D}(x,t) dx dt \\
% & = \frac{T^3}{R^2}
%  \frac{2 \beta^2}{\pi}
%  \int_0^1 q  \int_0^1
%    \int_{\R}  \int_{\R}
%    \int_0^{2 \pi} \int_0^{2\pi}
%  e^{ -\frac{\beta}{\tau} \left( q^2 + r^2 - 2 q r \cos(\theta_1) +   q^2 + s^2 -
%  2 q s \cos(\theta_2)
%  \right)} \\
%  & \qquad \qquad
% \frac{\partial}{\partial r} \left( r \frac{\partial}{\partial r}  g^{(1)}(r) \right)
%  \frac{\partial}{\partial s} \left( s \frac{\partial}{\partial s}  g^{(2)}(s) \right)
%    d \theta_1 d \theta_2   dr  ds d \tau d q  \\
& =   \frac{T^3}{R^2}
 8 \pi \beta^2
 \int_0^1 q  \int_0^1
   \int_{\R}  \int_{\R}
 e^{ -\frac{\beta}{\tau}\left( q^2 + r^2 +   q^2 + s^2 \right)} I_0(\tfrac{2 \beta q r}{\tau})
 I_0(\tfrac{2 \beta q s}{\tau})
 \\
 & \qquad \qquad
\frac{\partial}{\partial r} \left( r \frac{\partial}{\partial r}  g^{(1)}(r) \right)
 \frac{\partial}{\partial s} \left( s \frac{\partial}{\partial s}  g^{(2)}(s) \right)
   d \theta_1 d \theta_2   dr  ds d \tau d q \, .
\end{align*}

We next assume that $g^{(1)}$ and $g^{(2)}$ have compact support and
we integrate by parts with respect to the derivatives
$\frac{\partial}{\partial r}$ and $\frac{\partial}{\partial s}$
observing that by radial symmetry the terms at $r = 0$ and $s = 0$
vanish. Thus,
\begin{align*}
&\int_{\|x\|\leq R} \int_0^T  \tfrac{\partial u^{(1)}}{\partial D}(x,t)
 \tfrac{\partial u^{(2)}}{\partial D}(x,t) dx dt \\
& = \frac{T^3}{R^2}
  8 \pi \beta^2
 \int_0^1 q  \int_0^1
   \int_{\R}  \int_{\R}
 \frac{\partial^2}{\partial r \partial s}
 \left[
  e^{ -\frac{\beta}{\tau}\left( q^2 + r^2 +   q^2 + s^2 \right)} I_0(\tfrac{2 \beta q r}{\tau})
 I_0(\tfrac{2 \beta q s}{\tau})
 \right] \\
 & \qquad \qquad
 r s \frac{\partial}{\partial r}  g^{(1)}(r)
\frac{\partial}{\partial s}  g^{(2)}(s)
  dr  ds d \tau d q  .
\end{align*}
The $r,s$-derivatives can be calculated using the fact for modified Bessel functions that $I_0' = I_1.$
\begin{align*}
 &\frac{\partial^2}{\partial r \partial s}
 \left[
  e^{ -\frac{\beta}{\tau}\left( q^2 + r^2 +   q^2 + s^2 \right)} I_0(\tfrac{2 \beta q r}{\tau})
 I_0(\tfrac{2 \beta q s}{\tau})
 \right] \\
 & =   e^{ -\frac{\beta}{\tau}\left( q^2 + r^2 +   q^2 + s^2 \right)} \times \\
& \qquad  \left\{\tfrac{- 2 r \beta}{\tau} I_0(\tfrac{2 \beta q r}{\tau})  + \tfrac{2 q \beta}{\tau}I_1(\tfrac{2 \beta q r}{\tau})
 \right\}
  \left\{\tfrac{- 2 s \beta}{\tau} I_0(\tfrac{2 \beta q s}{\tau})  + \tfrac{2 q \beta}{\tau}I_1(\tfrac{2 \beta q s}{\tau})
 \right\},
 \end{align*}
which with $\rho = \frac{\beta}{\tau}$  finally leads to
\begin{equation}\label{final1}
\begin{split}
&\int_{\|x\|\leq R} \int_0^T  \tfrac{\partial u^{(1)}}{\partial D}(x,t)
 \tfrac{\partial u^{(2)}}{\partial D}(x,t) dx dt \\
% & = \frac{T^3}{R^2}
%   8 \pi \beta^2
%  \int_0^1 q  \int_0^1
%    \int_{\R}  \int_{\R}    e^{ -\frac{\beta}{\tau}\left( q^2 + r^2 +   q^2 + s^2 \right)}  \times   \\
%   & \qquad \qquad
%  \left\{\tfrac{2 r \beta}{\tau} I_0(\tfrac{2 \beta q r}{\tau})  - \tfrac{2 q \beta}{\tau}I_1(\tfrac{2 \beta q r}{\tau})
%  \right\}
%   \left\{\tfrac{ 2 s \beta}{\tau} I_0(\tfrac{2 \beta q s}{\tau})  - \tfrac{2 q \beta}{\tau}I_1(\tfrac{2 \beta q s}{\tau})
%  \right\} \\
%  & \qquad \qquad
%  r s \frac{\partial}{\partial r}  g^{(1)}(r)
% \frac{\partial}{\partial s}  g^{(2)}(s)
%   dr  ds d \tau d q   \\
&   = \frac{T^3}{R^2}
  32 \pi \beta^3
 \int_0^1 q  \int_\beta^\infty
   \int_{\R}  \int_{\R}    e^{ - \rho \left( q^2 + r^2 +   q^2 + s^2 \right)}  \times   \\
  & \qquad \qquad
 \left\{ r  I_0(2 q r \rho )  -  q  I_1(2  q r \rho )
 \right\}
  \left\{ s   I_0(2 q s \rho )  -  q  I_1(2  q s \rho )
 \right\} \\
 & \qquad \qquad
 {r s} \frac{\partial}{\partial r}  g^{(1)}(r)
\frac{\partial}{\partial s}  g^{(2)}(s)
  dr  ds d \rho  d q.   \\
\end{split}
  \end{equation}

For later reference we define a kernel as follows:
\begin{equation}\label{defker}
\begin{split}
k(r,s) &:=
 \int_0^1 q  \int_\beta^\infty    e^{ - \rho \left( q^2 + r^2 +   q^2 + s^2 \right)}  \times   \\
  & \qquad \qquad
 \left\{ r  I_0(2 q r \rho )  -  q  I_1(2  q r \rho )
 \right\}
  \left\{ s   I_0(2 q s \rho )  -  q  I_1(2  q s \rho ) \right\}  {r s} d \rho  d q.
 \end{split}
\end{equation}

For the numerical calculations, we also require the kernel for the case of
$\beta = 0$. We present another  formula which avoids integration over the
infinite domain with respect to $\rho$.

\begin{lemma}
For $\beta = 0$ the kernel $k(r,s)$ can be written as
\begin{equation}\label{fff}
\begin{split} &k(r,s) = \\
&\frac{1}{4 \pi^2}
\int_0^1  \int_0^{2 \pi} \int_0^{2 \pi}  q r s
\frac{ \left( r  -  q   \cos(\theta_1) \right) \left( s -  q   \cos(\theta_2) \right)}
{q^2 + r^2  - 2 q r \cos(\theta_1) +   q^2 + s^2
- 2 q s \cos(\theta_2) }
d\theta_1 d \theta_2  d q
\end{split}
\end{equation}
\end{lemma}
\begin{proof}
Setting $\beta = 0$ and \eqref{bessel}
and
\[  I_1(a) = \frac{1}{2 \pi}  \int_0^{2 \pi} e^{a \cos(\theta)} \cos(\theta) d \theta, \]
we observe  that
\begin{align*}
&k(r,s) = \frac{1}{4 \pi^2}
\int_0^1 q  \int_0^\infty  \int_0^{2 \pi} \int_0^{2 \pi}   e^{ - \rho \left( q^2 + r^2  - 2 q r \cos(\theta_1) +   q^2 + s^2
- 2 q s \cos(\theta_2) \right)}  \times   \\
  & \qquad \qquad
 \left( r  -  q   \cos(\theta_1) \right) \left( s -  q   \cos(\theta_2) \right) r s  d\theta_1 d \theta_2 d \rho  d q \\
 & = \frac{1}{4 \pi^2}
\int_0^1  \int_0^{2 \pi} \int_0^{2 \pi}  q r s
\frac{ \left( r  -  q   \cos(\theta_1) \right) \left( s -  q   \cos(\theta_2) \right)}
{q^2 + r^2  - 2 q r \cos(\theta_1) +   q^2 + s^2
- 2 q s \cos(\theta_2) }
d\theta_1 d \theta_2  d q  \, .
\end{align*}
$\blacksquare$
\end{proof}

Next we use Lemma~3.1 according to which  the optimal initial  conditions are
radially symmetric $\{0,1\}$-valued with compact support.  This
means that in this case, the functions $g^{(i)}$ attain the value
$1$ on a  number of intervals and $0$ else, and the associated
initial condition is supported on  a  union of annuli and possibly a
disk.  We moreover assume that the function $g$, which
defines the initial conditions, has only jumps at $N \in \N$ places:
\begin{equation}\label{gequ}  g^{(i)}(r) = \begin{cases} 1 &  r \in [r_j,r_{j+1}], j = 1,\ldots N, 0 \leq r_1 <r_{j-1} < r_{j} < r_{j-1},  \\
 0 &\text{ else } \end{cases} \end{equation}
It follows in the sense of distributions that
\begin{equation}\label{deri}
 \frac{\partial}{\partial r}  g^{(i)}(r)
= \sum_{j= 1, r_1 \not= 0}^{N} (-1)^{N-j+1} \delta_{r_j}(r),
 \end{equation}
 where $\delta$ denotes the Dirac-distribution. The case that $r_1 = 0$ means that $u_0$ has  support in
 a disk around $0$ and clearly, there is no jump in the derivative there, hence this must be excluded from the sum.
 Note that the signs are such that the
 outer radius has $(-1)$ and the following have alternating signs, which becomes  clear by drawing
 the graph of such a function.

By plugging in the formula \eqref{deri} in place of $g^{(i)}$
we obtain the following proposition.

\begin{proposition}\label{pr:one}
Suppose that the initial conditions $v_0(z) = g(|z|)$ are radially symmetric $\{0,1\}$-valued
with $g$ having finitely many jumps as in \eqref{gequ}.
Then the sensitivity is given by
\begin{equation}
S_{int} =  \frac{T^3}{R^2}
  32 \pi \beta^3 \sum_{j,k=1}^N (-1)^{k + j} k(r_j,r_k)
  \end{equation}
 with the kernel $k$ given by \eqref{defker}.
\end{proposition}
Note that we do not have to take into account if $r_1 = 0$ or not, because the kernel is $0$ if one of the
arguments is $0$. We remark that Proposition~\ref{pr:one} is also true for the case of $g$ having
infinitely many jumps; in this case we have to set $N=\infty$.

By Proposition~\ref{pr:one} we find that
\begin{itemize}
\item Problem 1 is equivalent to
\[ \max_{r_i, N}  \sum_{j,k=1}^N (-1)^{k + j} k(r_j,r_k) , \qquad 0 < r_1 < r_2 \leq \ldots r_N, N \geq 1. \]
\item Problem 2 is equivalent to
\begin{align*}  &\max_{r_i, N}  \sum_{j,k=1}^N (-1)^{k + j} k(r_j,r_k) , \qquad 0 < r_1 < r_1 \leq \ldots r_N, N \geq 1 \\
\intertext{under the constraint that}
&\sum_{k=N}^{2}  r_{k}^2 - r_{k-1}^2 = \frac{c_1}{\pi} \quad \text{ or }
r_1^2 = \frac{c_1}{\pi} \quad \text{ if } N = 1.
\end{align*}
\end{itemize}

%\begin{remark}\label{rem4} {\rm

\subsection{Shape (and topology) optimization of the bleached region}
\label{shape} We try to solve the above problems numerically.
Note that we have the radii of the annuli/disks as
variables and the number $N$ of them. As optimization routine we use
the simplest and least sophisticated one, namely grid search over
the radii restricting us to a finite number of annuli. In fact, we
managed to search for the case $N \leq 4$ and for radii restricted
to $r_i \leq 5$ on a sufficiently fine grid. We could not think of a
smarter algorithm, but there are good reasons to believe that there
is no fast algorithm available.
%In fact, Problem 1 can be rephrased similar as mentioned in
%Remark~\ref{rem2.2}.
%It can be shown that the binary constrained can be relaxed and the problem
%is equivalent to maximizing the senstivity under the constraint that
%$0 \leq u_0(x) \leq 1$.
In fact, by considering the relaxed problem and introducing  the initial-to-sensitivity operator $K,$
Problem~1 is equivalent to
\[ \|K u_0\|_{L^2(\Omega) \times [0,T]}^2 \to \max   \qquad  \text{under } 0 \leq u_0 \leq 1. \]
This is a convex {\em maximization problem} and it can be regarded as the problem of calculating the
operator norm $K: L^\infty \to L^2$ (together with an additional positivity constraint $u_0 \geq 0$).
It is known that calculating this operator norm is an NP-hard problem \cite{Ro,Tro}. This makes us believe
that our optimization problems are hard problems indeed and that there is no simple efficient algorithm
possible.

Let us mention that a classical method to tackle such shape optimization methods is the 
{\em level set method} \cite{OsFe}. This could be applied as well here, but the problem is 
that this and related methods can only find {\em local} maxima. Thus, as we want to find 
 {\em global} maxima, the level set method does not help much here.

%}
%\end{remark}

\section{Numerical calculations and results}
As mentioned in Section~\ref{shape}, we perform a grid search to
solve Problems~1 and~2. For this we have to calculate the kernel
$k(r,s)$ for different values of $r$ and $s$ and also, to understand the
parameter dependence of the results, for different values of
$\beta$. The results are calculated using Matlab. Depending on the
grid for $r,s,$ and $\beta,$ according to \eqref{defker}, we have to
calculate a double integral for each point on a 3-dimensional grid
on $(r,s,\beta)$. No doubt, that for this tasks the evaluation has
to be done efficiently in order to have a reasonable runtime. When
calculating the double integral in \eqref{defker}, we observe that
Matlab does not have an ``ArrayValued'' option for multidimensional
integrals and hence the code cannot be vectorized. This faces us
with the problem that the 2-dimensional integral has to be
calculated by for-loops over the grid, which, as is well-known, is
usually extremely slow. As a remedy, we rewrite the problem as
iterated integral and apply Matlab's ODE-solver (which can be
vectorized).

Indeed, from  formula \eqref{defker}, it follows that the kernel satisfies 
\begin{equation}\label{trick}
\begin{split}
\frac{\partial}{\partial \beta} k(r,s) &=
- \int_0^1 q e^{ - \rho \left( q^2 + r^2 +   q^2 + s^2 \right)}  \times   \\
  & \qquad \qquad
 \left\{ r  I_0(2 q r \rho )  -  q  I_1(2  q r \rho )
 \right\}
  \left\{ s   I_0(2 q s \rho )  -  q  I_1(2  q s \rho ) \right\}  {r s}  d q.
\end{split}
  \end{equation}
with initial condition
\[ k(r,s)|_{\beta = 0} = \eqref{fff} \]
In this way we can use the ODE-solver to handle the $\beta$-dependence and only have to
calculate the one-dimensional parameter-dependent integral \eqref{trick}, where both
steps can be vectorized. This gives a speed-up of a factor of about 10 compared to
the for-loop and double integral approach. The only bottleneck is the calculation of the
initial conditions in form of a 3-dimensional integral. However, this only has to be
done on a 2-dimensional $(r,s)$ parameter grid (compared to the initially 3-dimensional
parameter domain including $\beta$).

\subsection{Problem 1}
We try to calculate the (numerically) optimal configuration for
Problem 1. As explained above, we only have to consider radially
symmetric initial conditions. We cannot calculate all configurations
but we restrict us to the case of at most $N=4$ radii. This means we
consider initial shapes in the form of a disk ($N = 1$), an
annulus ($N = 2$),  an annulus with a disk inside ($N =
3$), or a double annulus ($N = 4$). Other configurations with higher
$N$ are excluded because of runtime considerations. However, the
solutions indicate a pattern, which makes us believe that we have
found the optimal configurations in the respective parameter setting
of $\beta$.

We also note that we only have to evaluate the kernel $k$ because
the sensitivity is only multiplied by a constant which only depends
on the parameters $R,T,D$. Thus, with these values fixed, optimizing
$k$ is the same as optimizing $S_{int}$.

Let us come to the details of our calculations:
We set up a grid for $r$ and $s$ in the interval $[0,5]$ with 100 equal-sized subintervals
(of length $0.05$). For $\beta$ we set up a grid in the interval  $[0,20]$ with 200 subintervals
of length $0.1$. Then the kernel is calculated on these gridpoints (exploiting symmetry in $r$ and $s$).
With the calculated kernel, we evaluate the functional for $r,s$ on the grid
for Problem 1 as given in Proposition~\ref{pr:one}
for the cases $N = 1,2,3,4$ and find the  configuration that has maximal value.
 Note that this testing requires
up to 4 for-loops and is thus quite slow.
The results are depicted in Figure~\ref{fig1}.

\begin{figure}[p]
\centering
\includegraphics[width=0.7\textwidth]{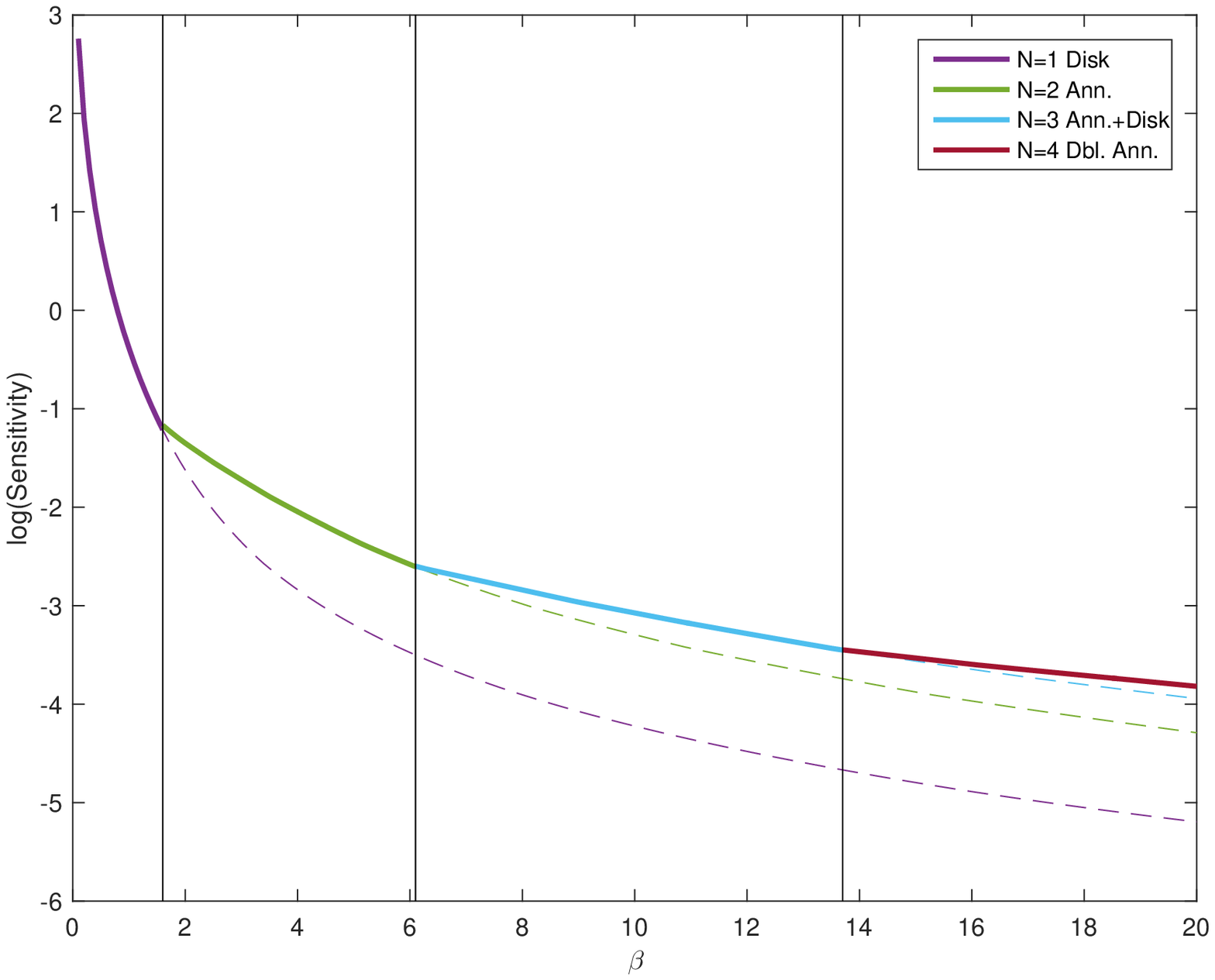}
\caption{Optimal configurations for different values of $\beta$. Tested were all configurations
with $N\leq 4$. Thick lines: values of the log of the optimal sensitivity. Dashed lines: values of the
log of the optimal sensitivity within each configurations (e.g. $N = 1, N=2. N= 3$).}\label{fig1}
\end{figure}

Let us explain the figure. The x-axis indicates  the values of $\beta$ while the $y$-axis corresponds
to the logarithm of the optimal value of the sensitivity for a fixed $\beta$.
The full lines in the figure display the log of the optimal value  over $\beta$  and the
associated  colors encode which of
the four possibilities ($N = 1,\ldots,4$) are the optimal configuration
(violet: $N=1$, green: $N = 2$, blue: $N = 3$, red: $N = 4$).

It can be observed that there are clear intervals for $\beta$  where
one specific configuration is always optimal. A value of $\beta$
where the shape of the optimal configuration changes (e.g., from
disk to annulus)  is indicated by a vertical line. Thus, we
observe  that from $\beta \in [0,1.8]$ the disk $(N = 1)$
is the optimal configuration, whereas for higher $\beta$ in the
interval $[1.8,6.1]$ the annulus is optimal. An even higher $\beta$
gives a annulus with a disk inside and for $\beta > 13.8$
we find the double annulus as optimal. Clearly a pattern can be
observed, namely that for increasing $\beta$ the number of
components of the optimal shape increases. We expect that for higher
$\beta$ not in the range of our calculations, more complicated
shapes (e.g. with $N = 5$) are optimal.  Although we tested only to
$N = 4$ the given pattern indicates that in the low range of $\beta,$
the shapes with few components are indeed optimal because in the
first interval $\beta \in [0,1.8]$ we never found  a better
structure with $N>1$.

In Figure~\ref{fig1} we also indicated the log-values of the optimal
sensitivity for each shape with fixed $N$ by a dotted line. For
instance, the dotted continuation of the violet line from $\beta =
1.8$ on corresponds to the optimal value within all disks
(i.e. $N = 1.$)  The continuation of the green line corresponds to
the optimal value within  all annuli and so on. It can be observed
that, for instance, in the range where the annulus is optimal, the
value of the sensitivity is significantly higher than the best
 disk. Roughly, the sensitivity of the best annulus is about
twice that of the best disk. This indicates that in
practice and for certain parameter setups, there can be a
significant gain of confidence when using an annulus instead of a
disk as initial bleach.

In the second figure, Figure~\ref{fig2}, we plotted the values of the radii of the optimal configuration
in each  interval of $\beta$. On the x-axis we put again the value of
$\beta$ and on the y-axis we have the scaled radius.
For instance in $\beta \in [0, 1.8]$ the disk is optimal
and only one radius is required for its parameterization. The value
of the optimal radius is given by the curve in the interval $[0,
1.8]$ In the next interval, the annulus is optimal which has an
outer and an inner radius as given by the two curves and so on, up
to the double annulus with 4 radii. Note that the jumps in this
graphs inside of the intervals are artefacts due to the discrete
grid for the radius; the exact curves should be smooth.

\begin{figure}[p]
\centering
\includegraphics[width=0.7\textwidth]{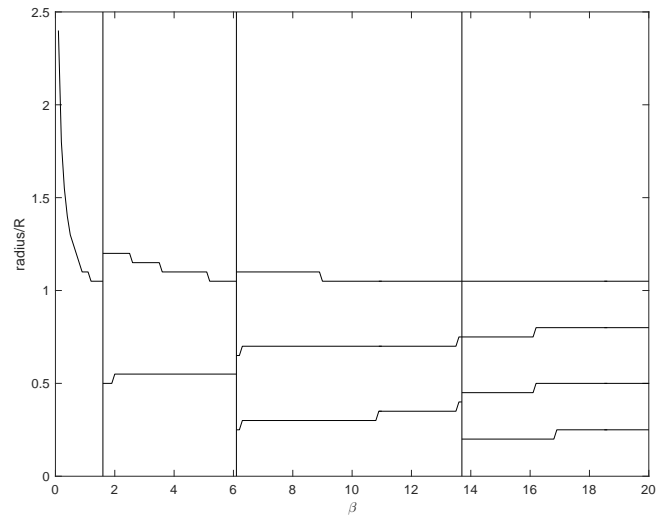}
\caption{Scaled radii of the optimal configurations as a function of $\beta$.} \label{fig2}
\end{figure}

\begin{figure}[p]
\centering
\includegraphics[width=0.7\textwidth]{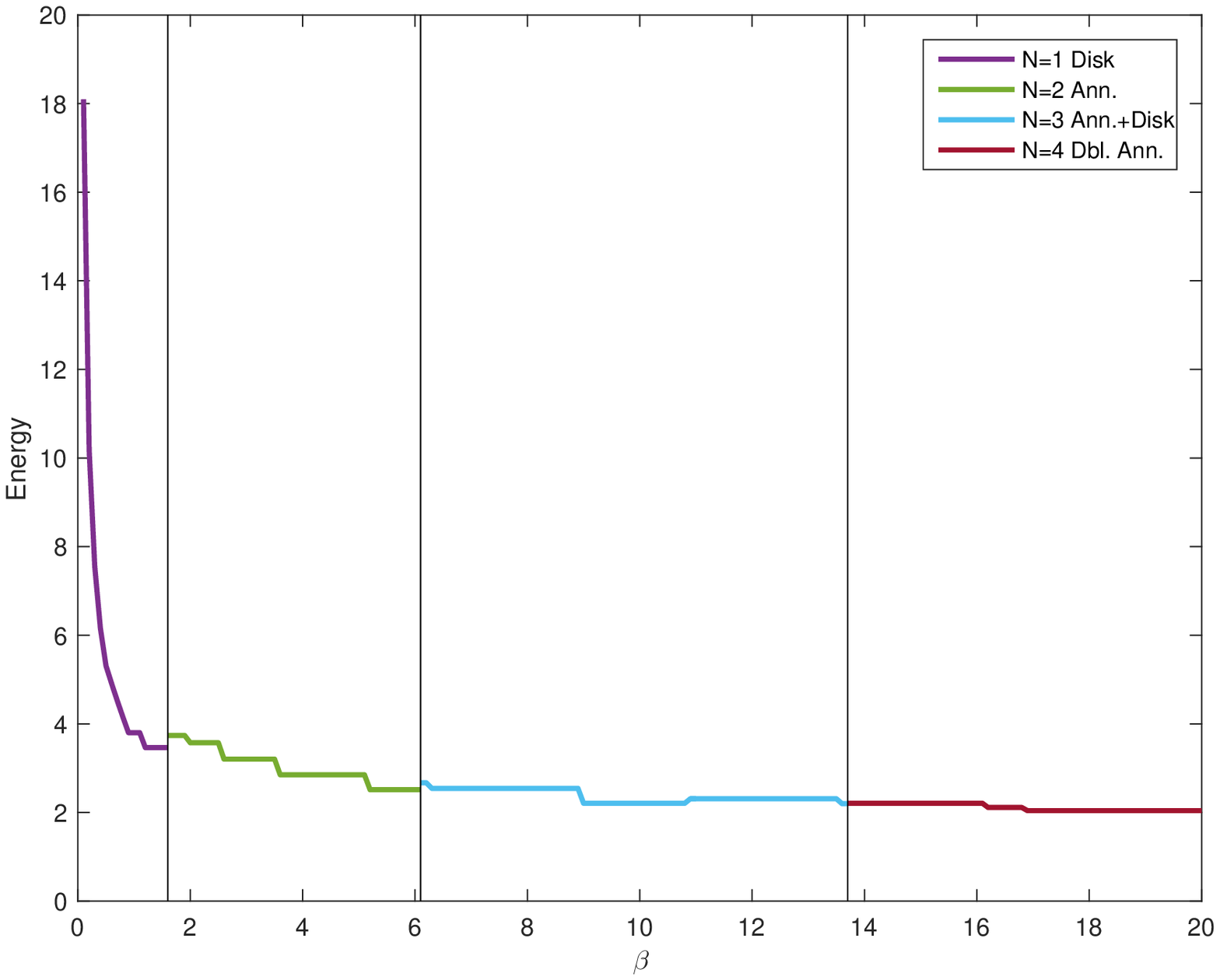}
\caption{Energy (i.e.,  $\int_{\R^2} u_0(x) dx$)  of the optimal configurations as a function
of $\beta$.} \label{fig2ex}
\end{figure}

In Figure~\ref{fig2ex}, we plotted the energy (i.e., the $L^1$-norm $\int_{\R^2} u_0(x) dx$) of the
optimal configurations as in Figure~\ref{fig2}. Of course, those can be easily calculated from the 
radii in Figure~\ref{fig2}.

For illustration purposes, in Figure~\ref{fig3}  we present some optimal initial shapes for some representative values of
$\beta$ in each of the intervals, for $\beta = 1,3,10,18$ (from left to right, from top to bottom).

\begin{figure}[p]
\centering
\begin{minipage}{0.49\textwidth}
\centering
\includegraphics[width=\textwidth]{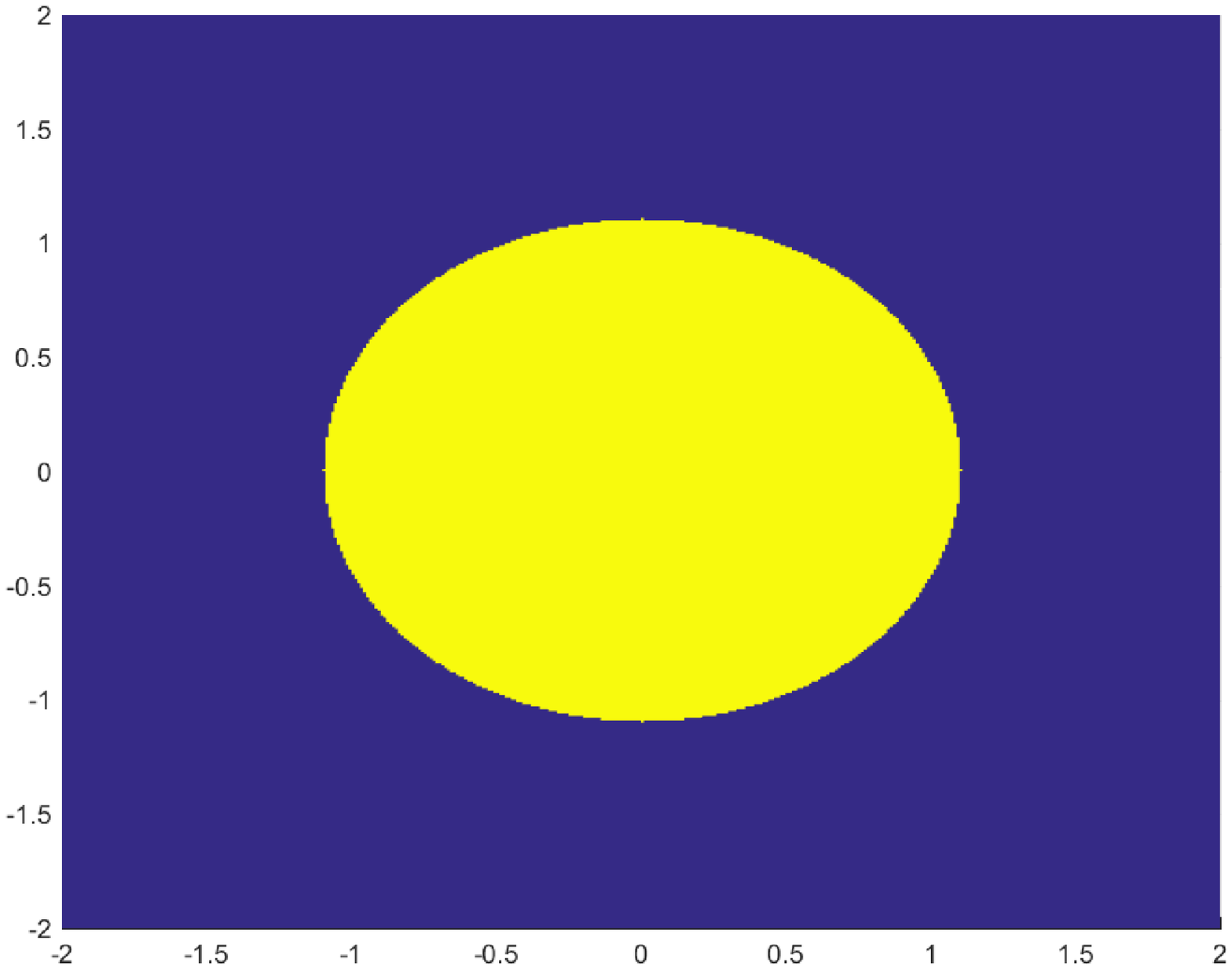}
{$\beta = 1$}

\includegraphics[width=\textwidth]{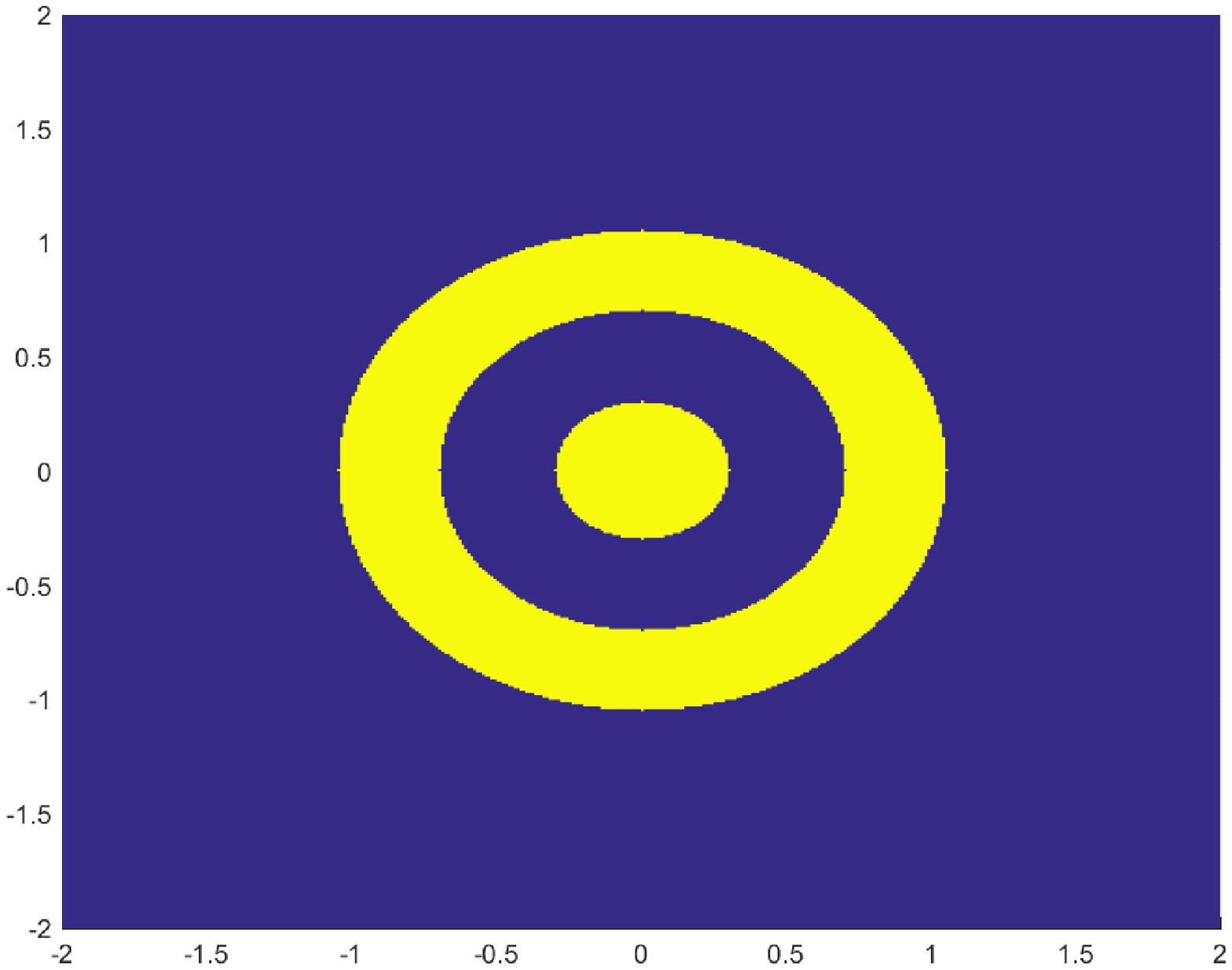}
{$\beta = 10$}
\end{minipage}
\begin{minipage}{0.49\textwidth}
\centering
\includegraphics[width=\textwidth]{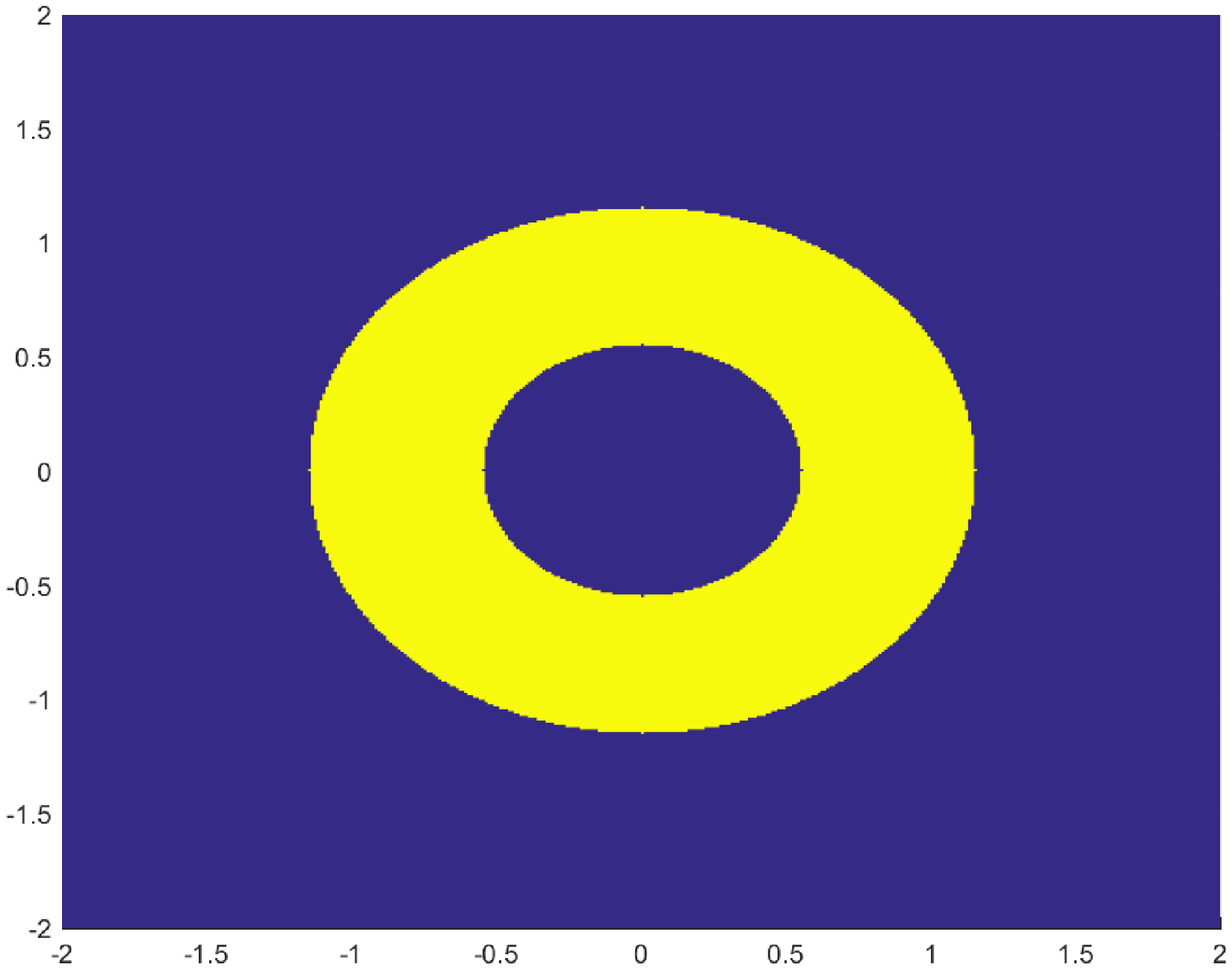}
{$\beta = 3$}

\includegraphics[width=\textwidth]{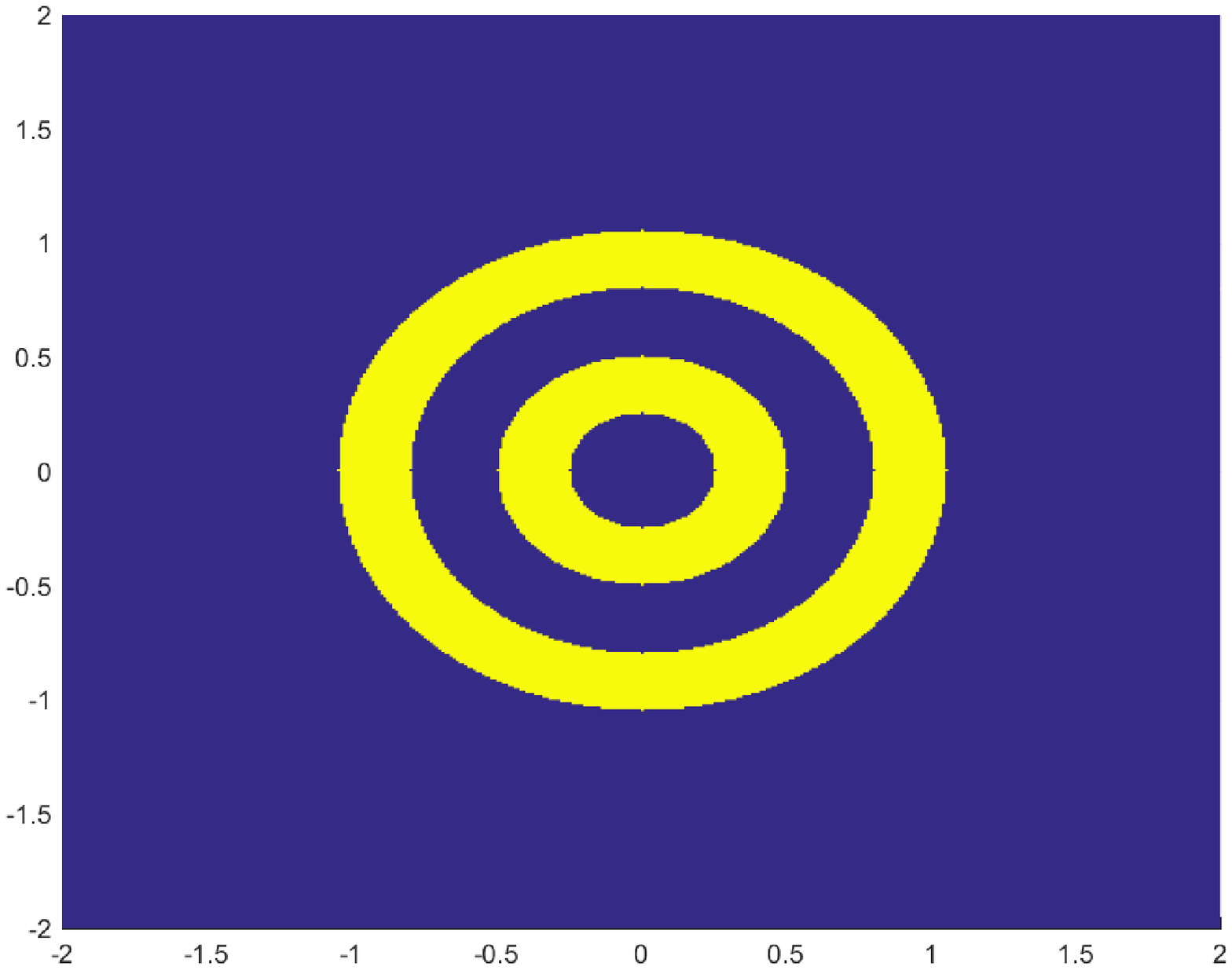}
{$\beta = 18$}
\end{minipage}
\caption{Optimal shapes of initial bleach for various values of $\beta$.}\label{fig3}
\end{figure}

\subsection{Problem 2}
The results from Problem 1 can also be used for  Problem 2.  Here we have the
additional restriction on the energy. The energy of some optimal configurations without constraints can 
be read off from Figure~\ref{fig2ex}. 

In order to find  an optimal shape, we again restrict
ourselves to $N \leq 4$. We calculate the value of the kernel for different configurations,
i.e., values of $N$, and $r_i$ and additional calculate their energy (the area of the initial bleach shape in this case).
%Then we plotted the kernel against the energy for each $N$.
For each of the four cases of $N$ we pick the configuration which is
optimal for each energy.  Note that this comes with a slight
difficulty because due to the discretization of the radii, the
energy values are discrete and in order to compare configurations,
we have to interpolate these values to a common energy grid. Note
that for the case $N = 1,$ when fixing the energy, the disk
is uniquely determined by the energy and there is no need to
optimize in this case because no degree of freedom is left.

\begin{figure}[p]\label{fig4}
\centering
\includegraphics[width=0.7\textwidth]{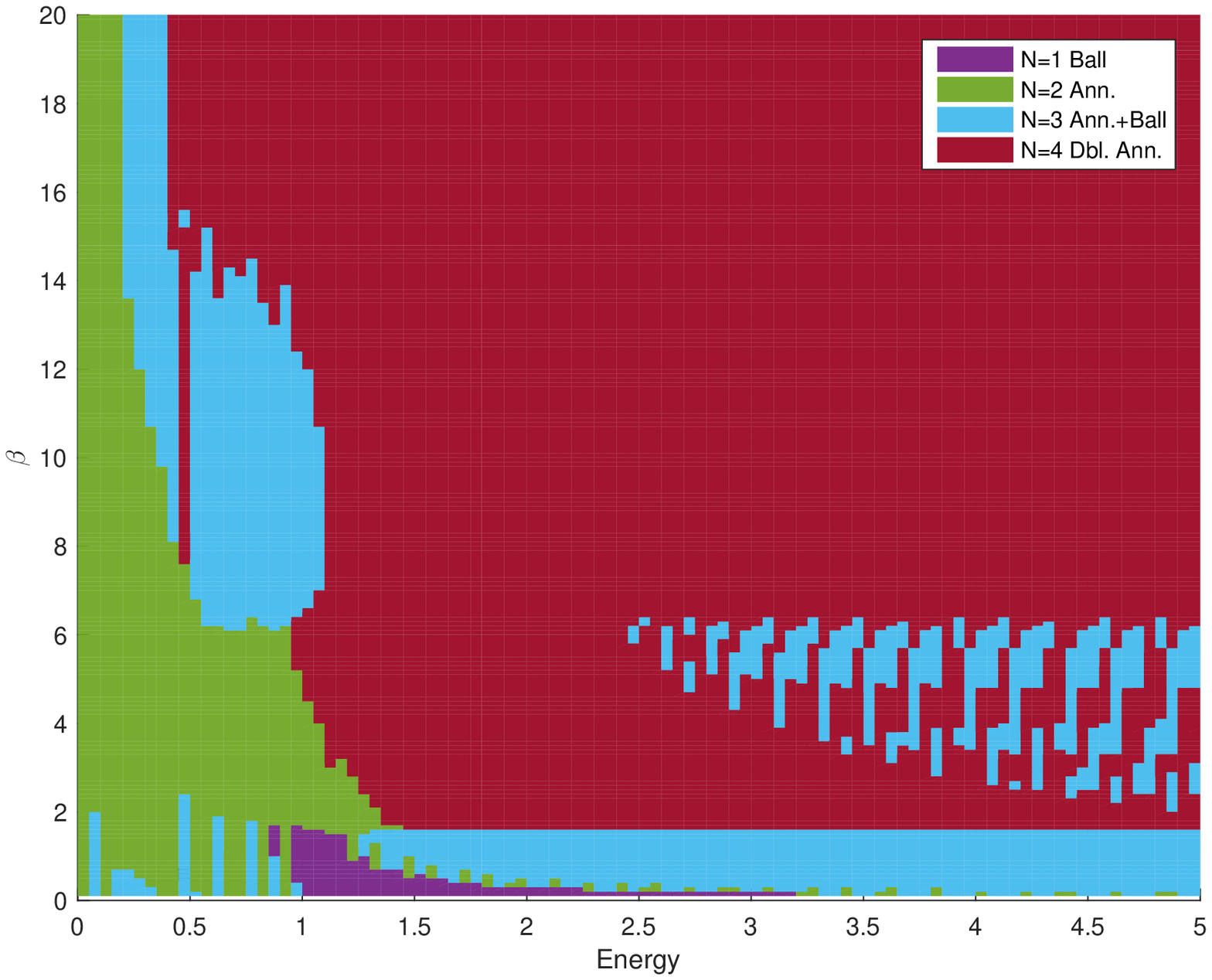}
\caption{Color-coded illustration of optimal configuration with $N \leq 4$ for different values of $\beta$ and
with fixed values of the energy (Problem 2). The colors indicates which configuration was found optimal.}
\end{figure}

In Figure~\ref{fig4}, we
display again a color-coded plot indicating the optimal configurations (up to $N\leq 4$).
For each value of the energy and each value of $\beta,$ the color at
the corresponding pixel indicates which of the configurations are
found optimal. We observe that in the majority of cases the double
annulus is optimal, but this is only because we restricted ourselves
to $N \leq 4$. It is to be expected that initial shapes with more
annuli are optimal in general and the problem does not reveal an
obvious pattern. Furthermore, we observe that the plot has  quite
irregular regions with not always clearly defined boundaries. This
can be an artefact of the discretization and we consider the plot
only as rough indication of what is going on.

Only in a tiny region around $\beta \sim 0$ energy $\sim 2,$ we have
a situation where the disk is optimal, but it is uncertain
if this is a valid result or an artefact from our discretization and
the interpolation of the energy values.  As a summary, it is quite
likely that simple-shaped objects (i.e., those with small  $N$) are
rarely (or even never) optimal for Problem~2.

Although, we have a put some model assumptions like symmetric 
observation domain and simple diffusion with constant parameter, 
we think that our results are relevant and---with modification---
valid also in more general cases. In particular, when the observational 
domain is not a disk but an square, we expect the optimal bleach shape for Problem~1
being a rounded square or several nested rounded squares, where the connectivity
again increases with $\beta$. For non-constant diffusion, the sensitivity 
is represented by a matrix and the optimal experimental design 
leads to the problem of maximizing some matrix norm of the inverse of the 
Fisher information matrix. Then, the optimal bleach shape for Problem~1 might depend on 
several additional design setups like the basis representation of the 
diffusion coefficient. Still it is not unlikely, that even in this case 
a similar pattern of optimal shapes with increasing connectivity 
emerges.

\section{Conclusion}
Our study started with the question: How does the bleach shape (and
topology) influence the accuracy of resulting parameter estimates?
Then the problem of the optimal initial shape for the identification
of a constant diffusion parameter was formulated.
 As
optimality criterion we choose to maximize a sensitivity measure in
order to have the expected error minimal; cf. \eqref{sensi}.
We studied
two problems, the first
%without additional restrictions
 with the fixed bleach depth
 and the
second with an fixed energy as an additional restriction.
We found out the analytical expressions for the sensitivity measure
$S_{int}$, cf. Proposition 3.2, allowing to compare different initial
bleach shapes.
Our numerical calculations revealed rather surprising results.
For  the first problem without restriction on energy, a clear
pattern is revealed. For small values of the scaled inverse
diffusion coefficient, the disk is the optimal shape and
for higher values, shapes with more and more components (i.e.
annuli-type shapes) become optimal. In particular, the disk
is not always the best shape. For practically relevant values of the
parameters,  sometimes an annulus can be better leading  to a
significant improvement in the confidence intervals.
For the later problem with restriction on energy, is seems to have
in most of the cases or even always only  highly
oscillating solutions.  
However, our ongoing research is directed to
this problem.

\subsection*{Acknowledgement}
  This work was supported  by
the Ministry of Education, Youth and Sports of the Czech Republic --
projects CENAKVA (No. CZ.1.05/2.1.00/01.0024), the Cenakva Centre
Development CZ.1.05/2.1.00/19.0380, and
 by the OeAD (Austrian agency for
international mobility and cooperation in education, science and
research) within the programme "Aktion Oesterreich-Tschechien
(AOeCZ-Universitaetslehrerstipendien)".
%
% ---- Bibliography ----
%

\end{document}